\documentclass{article}
\usepackage{anysize}
\marginsize{1in}{1in}{1in}{1in}
\usepackage[utf8]{inputenc}
\usepackage{bbold}
\title{Arithmetic functions that remain constant on strings of integers}
\author{Noah Lebowitz-Lockard \\
Department of Mathematics \\
University of Texas at Tyler, Tyler, TX \\
nlebowitzlockard@uttyler.edu \\
\\
J.C. Saunders \\
Department of Mathematical Sciences \\
Middle Tennessee State University, Murfreesboro, TN \\
John.Saunders@mtsu.edu}

\usepackage{amssymb}
\usepackage{amsmath}
\usepackage{amsthm}
\usepackage{mathtools}
\theoremstyle{definition}
\newtheorem{defn}{Definition}
\newtheorem{theorem}{Theorem}
\newtheorem{lemma}{Lemma}

\newtheorem{note}{Note}
\newtheorem{remark}{Remark}
\theoremstyle{definition}
\begin{document}
\maketitle
\begin{abstract}
In 2023, the first author and Vandehey \cite{Lebowitz-Lockard} proved that the largest $k$ for which the string of equalities $\lambda(n+1)=\lambda(n+2)=\cdots=\lambda(n+k)$ holds for some $n\leq x$, where $\lambda$ is the Carmichael $\lambda$ function, is bounded above by $O\left((\log x\log\log x)^2\right)$. Their method involved bounding the value of $\lambda(n + i)$ from below using the prime factorization of $n + i$ for each $i \leq k$. They then used the fact that \emph{every} $\lambda(n + i)$ had to satisfy this bound. Here we improve their result by incorporating a reverse counting argument on a result of Baker and Harman \cite{baker} on the largest prime factor of a shifted prime. 
\end{abstract}
\section{Introduction}
For a given arithmetic function $f$, let $F_f(x)$ be the largest $k$ for which the set of equalities $f(n+1) = f(n+2) = \cdots = f(n+k)$ has a solution satisfying $n+k\leq x$. In addition, let $G_f(x)$ be the largest $k$ for which the set of inequalities $f(n+1)\geq f(n+2)\geq\cdots\geq f(n+k)$ has a solution satisfying $n+k\leq x$.

The functions $F_f$ and $G_f$ have been studied for various functions $f$. Erd\H{o}s \cite{erdos} conjectured that $F_{\varphi}(x)\rightarrow\infty$ as $x\rightarrow\infty$, where $\varphi$ is Euler's totient function. To date, however, the only known solution to the equation $\varphi(n+1)=\varphi(n+2)=\varphi(n+3)$ is $n=5185$. Pollack, Pomerance, and Trevi\~{n}o \cite[Thm. $1.5$]{pollack} found an asymptotic formula for $G_\varphi (x)$.

\begin{theorem}[\cite{pollack}] As $x \to \infty$, we have
\[G_\varphi (x) \sim \log_3 x/\log_6 x,\]
where (here and below) $\log_k x$ refers to the $k$th iterate of the logarithm.
\end{theorem}

There are also results for other arithmetic functions as well. By modifying the proof of the previous theorem, one can also show that $G_\sigma (x) \sim \log_3 x/\log_6 x$, where $\sigma$ is the sum-of-divisors function. Sp\u{a}taru \cite{spataru} and the first author and Vandehey \cite{Lebowitz-Lockard} independently proved that $F_d(x)=\exp(O(\sqrt[3]{\log x\log_2x}))$, where $d$ is the number of divisors function. The first author and Vandehey \cite{Lebowitz-Lockard} also showed that $G_d(x) = O(\sqrt{\log x\log_2x})$. Their proofs relied on bounding the size of $d(n+1), \ldots, d(n+k)$ from below using the prime factorization of this common size.

In this note, we extend these results to the Carmichael $\lambda$ function, which we define below.

\begin{defn} The \emph{Carmichael function} $\lambda(n)$ refers to the smallest number $m$ for which the congruence $a^m \equiv 1 \textit{ mod } n$ holds for all $a$ coprime to $n$.
\end{defn}

The Fermat-Euler Theorem implies that $\lambda(n) \leq \varphi(n)$ for all $n$. Carmichael first defined this function in 1910 \cite{carmichael1, carmichael2}. He also found a simple formula for computing $\lambda(n)$.

\begin{theorem} For all $n$, we have
\[\lambda(n) = \left\{\begin{array}{ll}
\varphi(n), & \textrm{if } 8 \nmid n, \\
\varphi(n)/2, & \textrm{if } 8 | n.
\end{array}\right.\]
\end{theorem}

Fermat's Little Theorem states that for a given prime $p$, we have $a^{p - 1} \equiv 1 \textrm{ mod } p$ for all non-multiples $a$ of $p$. In particular, for a given prime $p$, we have $\lambda(p) = p - 1$. The number $n$ is an \emph{Carmichael number} if it is composite, but still satisfies $\lambda(n) | n - 1$. Alford, Granville, and Pomerance \cite{alford} showed that there are infinitely many Carmichael numbers. Last year, Larsen proved that for all $C > 1/2$, there is a Carmichael number in the interval $[x, x + x/(\log x)^C]$ for all sufficiently large $x$. (For a survey of results on Carmichael numbers, see \cite{pomerance}.)

The first author and Vandehey also found an alternative proof that $F_\lambda (x)=\exp(O(\log x\log\log x)^2)$. In this note, we obtain a better bound for $F_\lambda(x)$ by incorporating a result of Baker and Harman \cite{baker} on the largest prime factor of a shifted prime.

\begin{theorem}\label{thm1}
As $x\rightarrow\infty$, $F_{\lambda}(x)=O\left((\log x)^{1/0.677}\right)$. 
\end{theorem}

\begin{note}
For notational convenience, we let $c = 0.677$ from this point on. The quantity $0.677$ in the previous theorem is not exact and refers to the exponent in \cite[Thm. $2$]{baker}.
\end{note}

\section{Proof}
Our proof begins with the following result of Baker and Harman \cite{baker}, quoted as Theorem $1$ in \cite{harman}.
\begin{note}
Throughout the rest of the paper we let $p$ and $q$ denote prime values.
\end{note}
\begin{lemma}\label{lem1}
For every $a\in\mathbb{Z}$ and $0<\theta\leq c$ there exists $0<\delta(\theta)<1$ such that for sufficiently large $x>X(a,\theta)$ we have
\begin{equation*}
\sum_{\substack{p\leq x\\P(p+a)>x^{\theta}}}1>\delta(\theta)\frac{x}{\log x},
\end{equation*}
where $P(n)$ is the largest prime factor of $n$.
\end{lemma}
We use Lemma $1$ to derive the following result.
\begin{lemma}\label{lem2}
There exists a constant $C>0$ such that for sufficiently large $x$, we have
\begin{equation*}
\frac{x^{0.677}}{\log x}\leq C\cdot\#\{q:q>x^c,\exists p\leq x\text{ such that }p\equiv -1\pmod q\}.
\end{equation*}
\end{lemma}
\begin{proof}
Fix a positive integer $a$ and some $\theta \leq c$. The previous lemma implies that there exists $\delta := \delta(\theta) \in (0, 1)$ such that if $x$ is sufficiently large, then
\begin{equation*}
\sum_{\substack{p\leq x\\P(p+a)>x^{\theta}}}1>\delta \frac{x}{\log x}.
\end{equation*}
In particular, setting $a=-1$ and $\theta=c$, we may choose $\delta$ so that
\begin{equation*}
\sum_{\substack{p\leq x\\P(p+1)>x^{c}}}1>\frac{\delta x}{\log x}.
\end{equation*}
for all sufficiently large $x$. Put another way, we have
\begin{equation*}
\#\{p \leq x : \exists q>x^{c}\text{ such that }p\equiv-1\pmod q\}>\frac{\delta x}{\log x}.
\end{equation*}
For any $p$ in the set that is the left-hand side of the above inequality the existence of $q$ is unique since if there were two such values of $q$, say $q_1$ and $q_2$, for any $p$ we would have $p+1>q_1q_2>x^{2c}=x^{1.354}>p^{1.354}$, a contradiction. Therefore, we may partition the above set as
\begin{equation*}
\bigcup_{x^{c<q<x}}\{p:p\leq x,p\equiv -1\pmod q\}
\end{equation*}
Therefore,
\begin{align*}
&\quad\#\{p:p\leq x,\exists q>x^{c},p\equiv-1\pmod q\}\\
&=\sum_{\substack{x^{c}<q\leq x\\ \exists p\leq x\text{ such that }p\equiv -1\pmod q}}\#\{p\leq x:p\equiv -1\pmod q\}.
\end{align*}
Since for each $q>x^{c}$
\begin{equation*}
\#\{p\leq x:p\equiv -1\pmod q\}\leq\left\lceil\frac{x+1}{q}\right\rceil\ll\frac{x}{q}<x^{1 - c},
\end{equation*}
we have
\begin{equation*}
\frac{x}{\log x}\ll\#\{q:q>x^{c},\exists p\leq x\text{ such that }p\equiv -1\pmod q\}\cdot x^{1 - c}.
\end{equation*}
Thus,
\begin{equation*}
\frac{x^{c}}{\log x}\ll\#\{q:q>x^{c},\exists p\leq x\text{ such that }p\equiv -1\pmod q\}. \qedhere
\end{equation*}
\end{proof}

Using this result, we can bound the length of a sequence of numbers with the same sum of divisors.

\begin{lemma}\label{lem3}
Let 
\begin{equation*}
T=\lambda(n+1)=\lambda(n+2)=\cdots=\lambda(n+k).
\end{equation*}
Let $C$ be the constant in Lemma $1$. Then,
\begin{equation*}
\exp\left(C\cdot c\cdot\left(\frac{k}{2}\right)^c\right)\leq T.
\end{equation*}
\end{lemma}
\begin{proof}
We may assume that $k$ is sufficiently large so that Lemma $1$ holds with $x=k/2$. Consider a prime $q>\left(\frac{k}{2}\right)^{c}$ such that there exists a prime $p\leq\frac{k}{2}$ with $p\equiv 1\pmod q$. Then there exists $1\leq i\leq k$ such that $p\|n+i$. Therefore, $p-1\mid T$, so that $q\mid T$. $T$ is therefore, bounded below by the product of all such $q$. Lemma $2$ therefore implies
\begin{equation*}
\exp\left(C\cdot c\cdot\left(\frac{k}{2}\right)^{c}\right)=\left(\frac{k}{2}\right)^{C\cdot c\cdot\frac{\left(\frac{k}{2}\right)^{c}}{\log k-\log 2}}\leq T. \qedhere
\end{equation*}
\end{proof}
We now show that Lemma \ref{lem3} directly implies Theorem \ref{thm1}.
\begin{proof}
By Lemma \ref{lem3}, there exists a constant $D>0$ such that
\begin{equation*}
T\geq\exp\left(Dk^{c}\right).
\end{equation*}
By Mertens' Theorem we have $T\ll x\log\log x$. Therefore, $k^{1/c}\ll\log x$. Thus, $k\ll(\log x)^{1/c}$.
\end{proof}
\begin{remark}
The above proof also works with $\lambda$ replaced with $\varphi$, $\sigma$, or $\sigma_d$, the sum of the $d$th powers of all the divisors function, where $d$ is any positive odd integer. The proof is identical in the case of $\varphi$. For $\sigma$ one just has to take $a=1$ in applying Lemma $1$ and replacing the congruence $p\equiv 1\pmod q$ in Lemma \ref{lem2} with $p\equiv -1\pmod q$. Of course, in both of these cases, the resulting bound is still much weaker than Pollack, Pomerance, and Trevi\~{n}o's result \cite{pollack} with their method and result also holding for $\sigma$. For $\sigma_d$, where $d$ is any positive odd integer, one simply makes the observation in the proof of Lemma 3 that $p\|n+i$ also implies that $p^d+1\mid T$, where $T=\sigma_d(n+i)$ so that $p+1\mid T$ since $p+1\mid p^d+1$. Unfortunately though this proof does not work for positive even values of $d$.
\end{remark}
\begin{remark}
The Elliott-Halberstam Conjecture \cite[pg. 403]{tenenbaum} implies that in Lemma \ref{lem1} we can increase the range of $\theta$ to $0<\theta<1$. If this is true, we can replace the exponent of $1/c$ in Theorem \ref{thm1} with $1+o(1)$.
\end{remark}

\end{document}